\documentclass[reqno,12pt]{amsart}
\usepackage[utf8]{inputenc}

\usepackage{amsthm,amsmath,amssymb,booktabs,xcolor,graphicx,tikz,enumitem, setspace}
\usepackage[margin=1in]{geometry}
\numberwithin{equation}{section}
\newtheorem{theorem}{Theorem}[section]
\newtheorem{proposition}[theorem]{Proposition}
\newtheorem{lemma}[theorem]{Lemma}

\newcommand{\lam}{\lambda}

\newtheorem{conjecture}[theorem]{Conjecture}
\newtheorem*{question*}{Question}

\theoremstyle{definition}
\newtheorem{definition}[theorem]{Definition}

\tikzstyle{P} = [draw, circle, black, fill, inner sep = 0pt, minimum width = 3pt]
\tikzstyle{every loop} = []

\newtheorem{question}[theorem]{Question}
\newtheorem{defn}[theorem]{Definition}

\theoremstyle{remark}
\newtheorem*{remark}{Remark}

\makeatletter
\def\resetMathstrut@{%
  \setbox\z@\hbox{%
    \mathchardef\@tempa\mathcode`\[\relax
    \def\@tempb##1"##2##3{\the\textfont"##3\char"}%
    \expandafter\@tempb\meaning\@tempa \relax
  }%
  \ht\Mathstrutbox@\ht\z@ \dp\Mathstrutbox@\dp\z@}
\makeatother
\begingroup
  \catcode`(\active \xdef({\left\string(}
  \catcode`)\active \xdef){\right\string)}
\endgroup
\mathcode`(="8000 \mathcode`)="8000

\title{The Graph Density Domination Exponent}
\date{\vspace{-15mm}}

\newcommand{\edge}{
\begin{tikzpicture}
   [scale=0.3,auto=left,every node/.style={circle, draw=black, fill=black, minimum size=1pt,inner sep=1pt}]
    \node (n1) at (0,0) {};
    \node (n3) at (2,0) {};
   \foreach \from/\to in {n1/n3}
    \draw (\from) -- (\to);
  \end{tikzpicture}
}
\newcommand{\tri}{
\begin{tikzpicture}
   [scale=0.3,auto=left,every node/.style={circle, draw=black, fill=black, minimum size=1pt,inner sep=1pt}]
    \node (n1) at (0,0) {};
    \node (n2) at (2,0) {};
    \node (n3) at (1,2) {};
   \foreach \from/\to in {n1/n2, n2/n3, n3/n1}
    \draw (\from) -- (\to);
  \end{tikzpicture}
}
\newcommand{\squ}{
\begin{tikzpicture}
   [scale=0.3,auto=left,every node/.style={circle, draw=black, fill=black, minimum size=1pt,inner sep=1pt}]
    \node (n1) at (0,0) {};
    \node (n2) at (2,0) {};
    \node (n3) at (2,2) {};
    \node (n4) at (0,2) {};
   \foreach \from/\to in {n1/n2, n2/n3, n3/n4, n4/n1}
    \draw (\from) -- (\to);
  \end{tikzpicture}
}
\newcommand{\squz}{
\begin{tikzpicture}
   [scale=0.3,auto=left,every node/.style={circle, draw=black, fill=black, minimum size=1pt,inner sep=1pt}]
    \node (n1) at (0,0) {};
    \node (n2) at (2,0) {};
    \node (n3) at (2,2) {};
    \node (n4) at (0,2) {};
   \foreach \from/\to in {n1/n2, n2/n3, n3/n4, n4/n1,n1/n3}
    \draw (\from) -- (\to);
  \end{tikzpicture}
}
\newcommand{\ptri}{
\begin{tikzpicture}
   [scale=0.3,auto=left,every node/.style={circle, draw=black, fill=black, minimum size=1pt,inner sep=1pt}]
    \node (n1) at (0,0) {};
    \node (n2) at (0,2) {};
    \node (n3) at (2,1) {};
    \node (n4) at (4,1) {};
   \foreach \from/\to in {n1/n2, n2/n3, n3/n4, n3/n1}
    \draw (\from) -- (\to);
  \end{tikzpicture}
}
\newcommand{\ptria}{
\begin{tikzpicture}
   [scale=0.3,auto=left,every node/.style={circle, draw=black, fill=black, minimum size=1pt,inner sep=1pt}]
    \node (n1) at (0,1) {};
    \node (n2) at (1,2) {};
    \node (n3) at (2,1) {};
    \node (n4) at (4,1) {};
    \node (n5) at (1,0) {};
   \foreach \from/\to in {n1/n2, n2/n3, n3/n4, n3/n1,n1/n5, n3/n5}
    \draw (\from) -- (\to);
  \end{tikzpicture}
}
\newcommand{\ptrib}{
\begin{tikzpicture}
   [scale=0.3,auto=left,every node/.style={circle, draw=black, fill=black, minimum size=1pt,inner sep=1pt}]
    \node (n1) at (0,1) {};
    \node (n2) at (1,1) {};
    \node (n3) at (2,1) {};
   \foreach \from/\to in {n1/n2, n2/n3}
    \draw (\from) -- (\to);
  \end{tikzpicture}
}
\newcommand{\ptric}{
\begin{tikzpicture}
   [scale=0.3,auto=left,every node/.style={circle, draw=black, fill=black, minimum size=1pt,inner sep=1pt}]
    \node (n1) at (0,1) {};
    \node (n2) at (1,1) {};
    \node (n3) at (2,1) {};
    \node (n4) at (3,1) {};
    \node (n5) at (4,1) {};
   \foreach \from/\to in {n1/n2, n2/n3, n3/n4, n4/n5}
    \draw (\from) -- (\to);
  \end{tikzpicture}
}

\newcommand{\sm}{\setminus}

\begin{document}
\author[Stoner]{Cynthia Stoner}\
\begin{abstract}
For graphs $G$ and $H$, what relations can be determined between $t(G,W)$ and $t(H,W)$ for a general graph $W$? We study this problem through the framework of the density domination exponent, which is defined to be the smallest constant $c$ such that $t(G,W)\ge t(H,W)^c$ for every graph $W$. This broad generalization encompasses the Sidorenko conjecture, the Erd\H{o}s-Simonovits Theorem on paths, and a variety of other statements relating graph homomorphism densities. We introduce some general tools for estimating the density domination exponent, and extend previous results to new graph regimes.
\end{abstract}
\maketitle
\section{Introduction}

Given the density of one graph $G$ in a target graph(on) $W$, what can we say about the density of another graph $H$? Many problems in extremal graph theory fit into this framework. In order to discuss these, we first introduce some basic definitions. 

\begin{definition}
Given two graphs $G,W$, a \textit{graph homomorphism} is a vertex map $\phi:V(G)\to V(W)$ such that, for every edge $(u, v)$ in $G$, the image $(\phi(u), \phi(v))$ is an edge in $W$. The number of such maps is denoted by $\hom(G,W)$. 
\end{definition}

For example, when $G=W=K_n$, every homomorphism $\phi: V(G)\to V(W)$ corresponds to a permutation of the vertices of the complete graph $K_n$, so $\hom(G,W)=n!$. 

\begin{defn}
For graphs $G$ and graphons $W$, define the \textit{graph density} or \textit{graphon density} $t(G,W)$ as follows. 

If $W$ is a graph, then
\[t(G,W)=\frac{\hom(G,W)}{|V(W)|^{|V(G)|}}.\]
If $W$ is a graphon, and the vertices of $G$ are labeled $v_1,\ldots, v_n$, $n=|V(G)|$, then
\[t(G,W)=\int_{[0,1]^{n}}\prod_{(v_i, v_j)\in E(G)}W(x_i,x_j)\prod_{i=1}^ndx_i.\]
\end{defn}

\begin{remark}
In the case that $W$ is a graph, $t(G,W)$ denotes the probability that a random map of vertices $V(G)\to V(W)$ is a homomorphism. We will sometimes replace $G$ with the graph itself for visual purposes; for instance, $t(\edge, W)=t(K_2,W).$
\end{remark}

Sidorenko's conjecture, which states that $t(H,W)\ge t(\edge, W)^{|E(H)|}$ for bipartite graphs $H$, has been the object of much interest in recent decades (for example, \cite{kim2016two}, \cite{lee2017sidorenko}). It states that, for all graph(ons) $W$, and bipartite graphs $H$, the inequality $t(H,W)\ge t(\edge, W)^{|E(H)|}$ holds. In other words, for any fixed bipartite $H$ and edge density $d\in (0,1)$, the homomorphism density of $H$ in $W$ is minimized when $W$ is the uniform graphon $W:=d$ (or in the graph case, a large pseudorandom graph $G(N,d)$). 

In this work, we examine a more general framework, which was introduced by Blekherman and Patel in \cite{blekherman2020threshold}, who dealt with the case when the target graph is a single edge. The main object of study will be the following metric, which is the density analogue of the homomorphism domination exponent studied in \cite{kopparty2011homomorphism}.  

\newpage
\begin{defn}\label{dde}
Given two simple graphs $G,H$, the \textit{density domination exponent} $\rho(G,H)$ is defined as follows:
\[\rho(G,H)=\sup_W\frac{\log t(H,W)}{\log t(G,W)}\]
where the supremum ranges across graphs or graphons $W$ with $t(G, W)\neq 0$.
\end{defn}

\begin{remark}
A lower bound $\rho(G,H)\ge c$ corresponds to a construction of some graph or sequence of graphs $W_i$ for which $\frac{\log t(H,W_i)}{\log t(G,W_i)}\to c$. An upper bound $\rho(G,H)\le c$ corresponds to a proof that $t(H,W)\ge t(G,W)^c$ for every graph $W$. Proving an equality $\rho(G,H)=c$, therefore, requires both citing a construction and proving the corresponding bound. Since they are equivalent, we will use the graph and graphon formulations of Definition \ref{dde} interchangeably. 
\end{remark}

Using the $\rho$ notation, Sidorenko's conjecture becomes $\rho(\edge, H)=E(H)$ for bipartite graphs $H$. Indeed, the statement itself is directly equivalent with $\rho(\edge, H)\le E(H)$, and $\rho(\edge, H)\ge E(H)$ follows from taking $W$ to be the constant $\frac{1}{2}$ graphon.  

The question of the relative densities of graphs and their blowups \cite{shapira2009density} or the $3$-uniform triforce hypergraph and hyperedge \cite{fox2020triforce} also fit into this general framework. The former, for instance, studies the value $\rho(K_{1,1,1}, K_{a,b,c})$ for some regimes of positive integers $a,b,c$. 

In Section 2, we introduce some fundamental properties of $\rho(\cdot, \cdot)$, and provide useful methods for proving the upper-bound inequalities (i.e., those of the form $\rho(H,G)\ge c$, which is equivalent to $t(H,W)\ge t(G,W)^c$ for all $W$.) We also give some constructions which will be useful for verifying the corresponding lower bounds. 

In Section 3, we discuss the special case $H=K_{1,t}$, where we can give exact answers in some regimes and one-sided bounds in the rest. Sections 4 and 5 also follow this structure. Section 4 studies the case where both $G$ and $H$ are paths or cycles, and finds exact bounds for all but a few regimes, and Section 5 does the same for complete and complete multipartite graphs. One particular example of interest is the case where $G,H$ are both paths. This case encapsulates both the Blakely-Roy inequality and the Erd\H{o}s-Simonovits Theorem. In Theorem \ref{paths}, we establish the value $\rho(P_m,P_n)$ in all cases except when $m>n$ are odd with $n+1\not| m+1$.

Finally, in Section 6, we discuss some of the remaining open questions, including some minimum unresolved cases from the previous sections. 

In addition to the methods developed in Section 2, we also employ a wide variety of other techniques and constructions to establish $\rho$ values in particular regimes. For instance, in Lemma \ref{entropy}, a weighted form of Shearer's entropy lemma is used on a certain generalization of bipartite graphs, while in the proof of Theorem \ref{br} we develop a fractional form of the Blakely-Roy inequality. In Theorem \ref{hyper}, we use a construction of Gowers and Janzer, along with some double counting, an alternate characterization of vector majorization proved in Lemma \ref{major}, and Proposition \ref{hold} to prove a result about the density domination exponent of complete multipartite graphs. 

Throughout, we will use the convention that $d_S=|\bigcap_{v\in S}N(v)|$, where $S\subset V(W)$ is a nonempty collection of vertices in $W$. 

\section{Basic Properties}
We state the following proposition without proof, as these facts follow directly from the definition of $\rho$.
\begin{proposition}
The following hold for all graphs $G,H,J$ and positive integers $n$. Here, $nG$ denotes the union of $n$ disjoint copies of $G$. 
\begin{enumerate}[label=(\alph*)]
\item $\rho(G,H')\ge \rho(G,H)\ge \rho(G', H)$ and $\rho(G,G')\ge 1$ for any subgraphs $H'\subset H, G'\subset G$.
\item $\rho(G,H)\rho(H,J)\ge \rho(G,J)$.
\item $\rho(nG,H)=n\rho(G,H)$.
\item $\rho(G,nH)=\frac{1}{n}\rho(G,H)$.
\end{enumerate}
\end{proposition}

In addition to these, our primary basic tool for the bound proofs will be the following consequence of Holder's inequality. Given a graph $G$ and extra disjoint graphs $S_1, \ldots, S_n$ which are connected to $G$ in a specified way and a nonnegative integer vector $\textbf{k}=\{k_1,\ldots, k_n\}$, let $G_{\{S_i\}}(\bf k)$ denote the graph obtained by, for each $i$, copying the graph $S_i$, $k_i$ times, and attaching each copy to $G$ in the same way that $S_i$ is attached. 

\begin{proposition}\label{hold}
For a nonnegative integer vectors $\bf x, \bf y, \bf z$ and positive integers $a,b,c$ such that $a\textbf{x}+b\textbf{y}=(a+b)\bf z$, it holds for any graph $G$, disjoint graphs $\{S_1,\ldots, S_n\}$ with implicit connections to $G$, and graphon $W$ that
\[t(G_{\{S_i\}}(\bf x),W)^{a}t(G_{\{S_i\}}(\bf y),W)^{b}\ge t(G_{\{S_i\}}(\bf z),W)^{a+b}.\]
\end{proposition}
\begin{proof}
Via the definition of graph density, we may compute
\begin{gather*}t(G_{\{S_i\}}(\bf x),W)= \\\int_{[0,1]^{|V(G)|}}\prod_{(v_i, v_j)\in E(G)} W(v_i, v_j)\prod_{\ell=1}^n(\int_{[0,1]^{|V(S_\ell)|}}\prod_{(v_i,v_j)\in E(G\cup S_\ell)\sm E(G)}W(v_i, v_j)\prod_{v_i\in V(S_\ell)}dv_i)^{x_\ell}\prod_{v_i\in V(G)}dv_i\end{gather*}
The product over $\ell$ is a nonnegative measurable function of the $v_i\in V(G)$. The result therefore follows from an application of Holder's inequality. 
\end{proof}

Now, we establish a useful existence criterion for $\rho(G,H)$. For this, we first prove the finiteness of this quantity when $H$ is a blowup of $G$. 

\begin{defn}
For a simple graph $G=(V,E), V=\{v_1,\ldots, v_n\}$ and a sequence of positive integers $B=\{b_1,\ldots, b_n\}$, the \textit{blowup} $B(G)$ is the graph $(V', E')$, where:
\begin{itemize}
\item $V'=\{v_{ij}, 1\le i\le n, 1\le j\le b_i\}$.
\item $B(G)$ has an edge between $v_{ij}$ and $v_{k\ell}$ if and only if $(v_i, v_k)\in E$. 
\end{itemize}
\end{defn}
\begin{proposition}
Let $B=\{b_1,\ldots, b_n\}$ be a sequence of positive integers, and let $G$ be a graph. Then $\rho(G,B(G))\le\prod_{i=1}^nb_i$.
\end{proposition}
\begin{proof}
It suffices to consider the case where only one vertex is blown up, since iterating this $n$ times will yield the desired inequality. Hence, assume $b_1=k>1, b_2=\ldots=b_n=1$. Let $G'$ denote the graph obtained by removing $v_1$ from $G$. Then, Proposition \ref{hold} yields, for any graphon $W$:
\[t(G,W)^k\le t(G',W)^{k-1}t(B(G),W)\le t(B(G),W).\]
Hence $\rho(G, B(G))\le k$ as desired. 
\end{proof}
Now we can show the following.
\begin{proposition}\label{cond}
If $G$ is nonempty, then the density domination exponent $\rho(G,H)$ is finite exactly when $\hom(H,G)>0$. 
\end{proposition}
\begin{proof}
First, suppose $\hom(H,G)$ is positive. Then $H$ is a subgraph of some blowup $B(G)$ of $G$; from which it follows that $\rho(G,H)\le \rho(G,B(G))$ is finite from the previous proposition.

Now, suppose $\hom(H,G)=0$. Then, take $W\to G$ as a certificate that $\rho(G,H)$ is infinite. We're done.
\end{proof}

Finally, we establish several lower bounds which will be useful for the rest of the paper. 

\begin{proposition}\label{lowb}
The following hold for any nonempty graphs $G,H$.
\begin{enumerate}[label=(\alph*)]
\item $\rho(G,H)\ge \frac{|E(H)|}{|E(G)|}$.
\item $\rho(G,H)\ge \frac{|V(H)|}{|V(G)|}$.
\item $\rho(G,H)\ge \frac{|V(H)|-1}{|V(G)|-1}$ if $G,H$ are connected.
\item $\rho(G,H)\ge \frac{|V(H)|-\alpha(H)}{|V(G)|-\alpha(G)}$, where $\alpha(G)$ is the independence number of $G$.
\end{enumerate}
\end{proposition}
\begin{proof}
We exhibit the four constructions which correspond to these lower bounds.
\begin{enumerate}[label=(\alph*)]
\item Take $W$ to be the constant graphon, $W=p, 0<p<1$. Then $t(G,W)=p^{|E(G)|}$, while $t(H,W)=p^{|E(H)|}$, yielding the desired ratio.
\item Take $W$ to be the graphon for which $W(x,y)=1$ if $x<\frac{1}{2}$ and $y<\frac{1}{2}$, and $W(x,y)=0$ otherwise. Then $t(G,W)=2^{-|V(G)|}$, while $t(H,W)=2^{-|V(H)|}$, yielding the desired ratio. 
\item Take $W$ to be the graphon for which $W(x,y)=1$ if $x, y<\frac{1}{2}$ or $x, y>\frac{1}{2}$, and $W(x,y)=0$ otherwise. Then $t(G,W)=2^{1-|V(G)|}$, while $t(H,W)=2^{1-|V(H)|}$, yielding the desired ratio. 
\item Take $W_i$ to be the graph corresponding to $K_{1,i}$ with a loop on the distinguished vertex, and consider the sequence $\{W_i\}.$ Then $t(G,W_n)=n^{\alpha(G)-|V(G)|+o(1)}$, while $t(H,W)=n^{\alpha(G)-|V(H)|+o(1)}$, yielding the desired ratio. 
\end{enumerate}
\end{proof}

We'll give an example which demonstrates these tools in use, and also shows that the four constructions in the previous proposition do not provide tight bounds for all pairs $(G,H)$. 

\begin{proposition}
For $G=\ptri$ and $H=\squ$, it holds that $\rho(G,H)=\frac{4}{3}$.
\end{proposition}
\begin{proof}
All of the constructions in Proposition \ref{lowb} yield a lower bound on $\rho(G,H)$ of $1$. We can improve this by considering the graphon equivalent to the following weighted, looped graph on $n$ vertices for large $n$. Let $n-1$ vertices have uniform edge and loop weight $\frac{1}{\sqrt{n}}$, and let the last vertex have a weight $1$ loop and be connected to all other vertices with edge weight $1$.  Taking $W$ to be the corresponding graphon, we find that $t(G,W)=\Theta(n^{\frac{3}{2}})$, while $t(H,W)=\Theta(n^2)$. Taking $n\to\infty$, this family of graphons shows that $\rho(G,H)\ge\frac{4}{3}$. 

Now we prove the upper bound. For this, we use the following chain of inequalities, valid for any graphon $W$:
\begin{align*}
t(G,W)^2&=t(\ptri, W)^2
\\ &\le t(\ptria, W)t(\ptrib, W)
\\ &\le t(\ptric, W)t(\ptrib, W)
\\ &\le t(\squ, W)t(\squ, W)^{\frac{1}{2}}
\\ &= t(H, W)^{\frac{3}{2}},
\end{align*}
which rearranges to $t(H, W)\ge t(G, W)^{\frac{4}{3}}$ as required. Here, the first inequality is due to Proposition \ref{hold}, the second follows from a subgraph property, and the last inequality is two applications of the Cauchy-Schwarz inequality. 
\end{proof}

\section{$H$ is a Star}
In \cite{alon1981number}, Alon found tight bound for the extremal values of the asymptotic log-ratio of the number of copies of a graph $H$ in a larger graph $G$, and the number of edges of $G$. We use some of these methods to prove an analogous result for the density domination exponent when $H=\edge$, and extend one direction of this inequality to the case $H=K_{1,t}$. 

First, we note the following, which is a consequence of Theorem 4 from \cite{sidorenko1994partially}.

\begin{proposition}\label{starry}
Let $G$ be a connected graph with $|V(G)|\le t+1$. Then $\rho(G, K_{1,t})=\frac{t}{|V(G)|-1}$. 
\end{proposition}
\begin{proof}
From Proposition \ref{lowb} part (c), it follows that $\rho(G, K_{1,t})\ge\frac{t}{|V(G)|-1}$. For the other direction, consider any spanning tree $T$ of $G$ and any graphon $W$. We have $t(T,W)\le t(K_{1,|V(G)|-1}, W)$ by the aforementioned theorem, and $t(K_{1,|V(G)|-1}, W)\le t(K_{1,t},W)^{\frac{t}{|V(G)|-1}}$ by Holder's inequality. Hence
\[\rho(G,K_{1,t})\le \rho(T,K_{1,t})\le \frac{t}{|V(G)|-1}\]
and we're done.
\end{proof}

For the regime where $G$ as more vertices than $H=K_{1,t}$, we first define the following extension of the delta function introduced in \cite{alon1981number}.

\begin{definition}
For a graph $G$ and positive integer $i$, let $\delta_i(G)=\max_S(|S|-i|N(S)|)$, where $S$ ranges across all subsets of $V(G)$. 
\end{definition}

\begin{theorem}
For all positive integers $t$ and graphs $G$ with $|V(G)|\ge t+1$
\[\rho(G, K_{1,t})\ge \frac{t+1}{|V(G)|-\delta_t(G)}\]
and equality holds when $t=1$. 
\end{theorem}

\begin{proof}
We handle first the general lower bound. For each positive integer $i$, take $W$ to be the weighted graph defined as follows.

Take a clique of size $\left\lfloor n^{\frac{t}{t+1}}\right\rfloor$ and connect one of its vertices to $n$ additional vertices. Then $t(K_{1,t},W)$ is $\Theta(n)$, while $t(G,W)=\Theta(n^{\frac{t+1}{|V(G)|-\delta_t(G)}})$; hence, this gives the desired lower bound. 

Next, we show that equality in fact holds when $t=1$. For this, we use the following two lemmas from \cite{alon1981number}.
\begin{lemma}
Let $H$ be any graph with $\delta_1(G)=0$. Then $H$ can be vertex-partitioned into edges and odd cycles. 
\end{lemma}

\begin{lemma}
For any graph $G$, there exists an independent set $I$ with $|I|-|N(I)|=\delta_1(G)$.
\end{lemma}

Now, take any graph $G$, and let $I$ be an independent set chosen such that $|I|-|N(I)|=\delta_1(G)$. Let $H$ be the subgraph of $G$ obtained by removing $I, N(I)$ from $G$. We first claim that $\delta_1(H)=0$. Indeed, suppose otherwise, and let $S\subset V(H)$ be a subset of vertices with $|S|>|N(S)|$. Then
\[|I\cup T|-|N(I\cup T)|=|I|+|T|-N(I)-N_H(T)>\delta_1(G),\]
a contradiction to the definition of $\delta_1$. 

Now, let $J$ denote the subgraph of $G$ defined by removing $H$, so that $J\cup H$ is a subgraph of $G$. We claim that every vertex in $N(I)$ can be matched to a unique vertex in $I$, so that the resulting edges are disjoint. Indeed, for this we need only check the Hall's Marriage lemma condition, that for any subset $R\subset N(I)$, we have $|N_J(R)|\ge |R|$. Suppose otherwise, that some $R$ satisfies the opposite inequality. Then  
\begin{align*}|I\sm N_J(R)|-|N(I\sm N_J(R))|&= |I|-|N_J(R)|-|N(I)|+2(|N(I)\sm N(I\sm N_J(R))|)
\\ &\ge \delta_1(G)+|R|-|N_J(R)|
\\ &>\delta_1(G),
\end{align*}
once again a contradiction. Therefore such a matching indeed exists, and it follows that:
\[\rho(G, \edge)\le \frac{1}{\frac{|V(H)|}{2}+|J|}=\frac{2}{|V(G)|+|J|-|I|}=\frac{2}{|V(G)|-\delta_1(G)}\]
as desired.
\end{proof}

\newcommand{\odd}{\text{ odd}}
\newcommand{\even}{\text{ even}}
\section{Paths and Cycles}
In this section, we take $P_n$ to be the path on $n$ edges, and $C_n$ to be the cycle on $n$ edges. 

\begin{theorem}\label{paths}
Let $m, n$ be positive integers. Then
\[\rho(P_m, P_n)=\begin{cases}
\frac{n}{m} & n>m,\text{ } n\even
\\ \frac{n}{m} & n>m, n\odd,\text{ } m\odd
\\ \frac{n+1}{m} & n \odd,\text{ } m\even
\\ \frac{n+1}{m+1} & n<m,\text{ } n\even
\\ \frac{n+1}{m+1}, & n<m, n\odd,\text{ } m\odd, n+1|m+1.
\end{cases}\]
\end{theorem}
\begin{proof}
The lower bounds all follow from the constructions in Proposition \ref{lowb}. In particular, the first and second follow from (a), the third follows from (d), and the fourth and fifth follow from (b). 

Now we turn our attention to the upper bounds. The second case, which is the Erd\H{o}s-Simonovits conjecture, was proved recently by Blekherman and Raymond in \cite{blekherman2020proof}, where they also discuss the older first case. Together, these cases extend the Blakely-Roy inequality $t(P_m)\ge t(P_1)^m$. 

Now we examine the third case. If $n\ge m-1$, then it follows from the inequality chain
\[t(P_n, W)\ge t(P_{n+1}, W)\ge t(P_m)^{\frac{n+1}{m}},\]
where the second inequality follows from the Erd\H{o}s-Simonovits conjecture. If $n<m-1$, consider an even number $N>m$ such that $n+1|N+1$. Then $P_N$ contains $\frac{N+1}{n}$ vertex-disjoint copies of $P_n$, so $t(P_n, W)\ge t(P_N, W)^\frac{n+1}{N+1}$ for any graphon $W$. So, using the $n\ge m-1$ case,
\[t(P_n, W)\ge t(P_N, W)^\frac{n+1}{N+1}\ge t(P_m, W)^\frac{n+1}{m}\]
as desired.

For the fourth case, we first prove a lemma:
\begin{lemma}
Suppose $a,b,x,y,z$ are positive integers such that $x, y$ are even and $ax+by=(a+b)z$ and let $W$ be any graph. Then
\[t(P_x,W)^at(P_y, W)^b\ge t(P_z, W)^{a+b}.\]
\end{lemma}
\begin{proof}
First, consider the special case $y=x+2, z=x+4, a=b=1$. In this case, the inequality becomes $t(P_x, W)t(P_{x+4}, W)\ge t(P_{x+2}, W)^2$, which is a special case of Proposition \ref{hold}. It follows that the sequence
\[\left\{\frac{t(P_4, W)}{t(P_2, W)}, \frac{t(P_6, W)}{t(P_4, W)}, \frac{t(P_8, W)}{t(P_6, W)}, \ldots\right\}\]
is nondecreasing. Therefore, when $y$ is even, and without loss of generality $x<y<z$, we may apply the following inequality chain.
\begin{align*}
(\frac{t(P_y,W)}{t(P_x,W)})^a &\le (\frac{t(P_y,W)}{t(P_{y-2},W)})^\frac{a(y-x)}{2}
\\ &=(\frac{t(P_y,W)}{t(P_{y-2},W)})^\frac{b(z-y)}{2}
\\ &\le (\frac{t(P_{y+2},W)}{t(P_y,W)})^\frac{b(z-y)}{2}
\\ &\le (\frac{t(P_z,W)}{t(P_y,W)})^a
\end{align*}
which rearranges to the desired result. When $y$ is odd, we may use $t(P_y, W)^2\le t(P_{y-1},W)t(P_{y+1}, W)$ to reduce to the even case, so the lemma is proved.
\end{proof}

Now we use this lemma. Note that $P_{2mn+2m+n}=P_{(2m+1)(n+1)-1}$ contains $2m+1$ vertex-disjoint copies of $P_n$. 
\begin{align*}
t(P_n, W)^{2m(m+1)}&=t(P_n, W)^{(m-n)(2m+1)}t(P_n, W)^{2mn+2m+n}
\\ &\ge t(P_{2mn+2m+n}, W)^{m-n}t(P_n, W)^{2mn+m+n}
\\ &\ge t(P_m, W)^{2m(n+1)}
\end{align*}
as desired.

Finally, the fifth upper bound follows from the fact that $P_m$ contains $\frac{m+1}{n+1}$ copies of $P_n$ when $n+1|m+1$. We're done. 
\end{proof}

\begin{remark}
The remaining case, $\rho(P_m, P_n)$ for $n<m$ both odd and $n+1\nmid m+1$, will require a different kind of construction. For example, $\rho(P_5, P_3)\ge \frac{7}{10}$ via the construction in \cite{kopparty2011homomorphism} beats the $\frac{n+1}{m+1}$ bound. 
\end{remark}

\begin{proposition}\label{cycles}
For positive integers $m, n$ such that $m>2n$, we have $\rho(C_m, C_{2n})=\frac{2n}{m}$.
\end{proposition}

\begin{proof}
The lower bound follows from the construction in Proposition \ref{lowb}(a).

For the upper bound, let $W$ be any graph, and let its adjacency matrix be $M$. Suppose the eigenvalues of $M$ are $\lam_1, \ldots, \lam_k$; they are all real because $M$ is nonnegative real symmetric. Then the desired inequality is equivalent with:
\[(\sum_{i=1}^k(\frac{\lam_i}{k})^m)^{\frac{n}{m}}\ge \sum_{i=1}^k (\frac{\lam_i}{k})^n\]
which follows from the monotonocity of the $L_p$ norm. Here, we have used the fact that $\sum_{i=1}^k \lam_i^k$ is the trace of $M^k$, which counts the number of length $k$ cycles in $W$.
\end{proof}

\begin{theorem}
Let $m, n$ be positive integers. Then if $m>2$,
\[\rho(C_m, P_n)=\begin{cases}\frac{n+1}{m} & n\le m-1
\\ \frac{n}{m-1} & n\ge m-1
\end{cases}\]
and if $n>1$,
\[\rho(P_m, C_{2n})=\frac{2n}{m}.\]
\end{theorem}

\begin{proof}
The lower bounds follows from (b), (c), and (a) respectively in Proposition \ref{lowb}, so we turn to the upper bounds.

We'll deal first with the second upper bound. Here, we have the following inequality chain:
\[t(C_{2n}, W)\ge t(C_{2mn}, W)^{\frac{1}{m}}\ge t(P_m, W)^{\frac{2n}{m}}.\]
The first inequality follows from Proposition \ref{cycles}, while the second holds since Sidorenko's conjecture is true for cycles. 

Now we consider the first upper bound. Suppose that $n\ge m-1$. If $n$ is even, then Proposition \ref{paths} yields the chain:
\[t(P_n, W)\ge t(P_{m-1}, W)^{\frac{n}{m-1}}\ge t(C_m, W)^{\frac{n}{m-1}}\]
since $P_{m-1}$ is a subgraph of $C_m$. If $n$ is odd, then there are two cases. If $m$ is even, then $m-1$ is odd and the proof follows as above. Otherwise, let $C_m'$ denote the graph obtained from making blowing up one vertex in $C_m$ by a factor of two; note that it contains $P_m$ as a subgraph. Then we have the following chain.
\begin{align*}
t(P_n, W)^{2m-2}&= t(P_n,W)^mt(P_n,W)^{m-2}
\\ &\ge t(P_m,W)^nt(P_{m-2},W)^n
\\ &\ge t(C_m', W)^nt(P_{m-2}, W)^n
\\ &\ge t(C_m, W)^{2n}
\end{align*}
where the last line is a special case of Proposition \ref{hold}. 
\end{proof}

Finally, we give some bounds on $\rho(C_k, C_n)$ for $n$ even, and also for $k<n$ both odd. 

\begin{theorem}
Let $n,k$ be positive integers with $2n\ge k\ge 3$. Then
\[\frac{2n-1}{k-1}\le \rho(C_k,C_{2n})\le \frac{2n-1-\frac{1}{2n-1}}{k-1-\frac{1}{2n-1}}<\frac{2n-1}{k-1}+\frac{1}{(k-1)^2}.\]
\end{theorem}
\begin{proof}
The lower bound follows from Proposition \ref{lowb}(c).

For the upper bound, we use the same setup and notation as in the proof of Proposition \ref{cycles}. We calculate, via Holder's inequality:
\begin{align*}
t(C_{2n},W)^{2nk-2n-k}&\ge t(C_{2n},W)^{2n(k-2)}t(K_2,W)^{2n(2n-k)}
\\ &=(\sum_{i=1}^k(\frac{\lam_i}{k})^{2n})^{2n(k-2)}(\sum_{i=1}^k(\frac{\lam_i}{k})^{2})^{2n(2n-k)}
\\ &\ge (\sum_{i=1}^k(\frac{\lam_i}{k})^{k})^{2n(2n-2)}
\\ &= t(C_k, W)^{2n(2n-2)}.
\end{align*}
Therefore $\rho(C_{2n}, C_k)\le \frac{2n(2n-2)}{2nk-2n-k}=\frac{2n-1-\frac{1}{2n-1}}{k-1-\frac{1}{2n-1}}$ as desired.
\end{proof}

For fixed $k$ and $n\to\infty$, the difference between lower and upper bounds approaches $\frac{1}{(k-1)^2}$; it would be interesting to close this gap.
\begin{question}
Is it true, for positive integers $n,k$ satisfying $2n\ge k\ge 3$, that $\rho(C_k, C_{2n})=\frac{2n-1}{k-1}+o_{k,n}(1)$?
\end{question}

Finally, we consider the case where $n>k$ and both are odd. 
\begin{theorem}\label{br}
Let $n>k$ be positive integers. Then
\[\frac{n}{k}\le \rho(C_{2k+1},C_{2n+1})\le \left\lceil\frac{n}{k}\right\rceil +1.\]
\end{theorem}
\begin{proof}
The lower bound again follows from construction (c) in Proposition \ref{lowb}.

For the upper bound, we make use of the following to lemmas. For nonnegative reals $0\le \alpha, \beta \le 1$ and a nonnegative integer $r$, define the \textit{generalized path} $P_{\alpha,r,\beta}$ to be a symbol which satisfies, for graphs $W$,
\[t(P_{\alpha,r,\beta},W)=\frac{1}{|V(W)|^{r+1+\alpha+\beta}}\sum_{v_0, \ldots, v_n\in V(W)\atop (v_i, v_{i+1})\in E(W)\forall i}d_{v_0}^\alpha d_{v_1}^\beta.\]
This can be thought of as a path with partial edges attached at each end, with weights $\alpha$ and $\beta$.

\begin{lemma}[Generalized Blakely-Roy inequality]\label{blakeroy}
For positive integers $r$, positive reals $0\le \beta\le 1$, and weighted graphs $W$, the following holds.
\[t(P_{0,n,\beta}, W)\ge t(P_1, W)^{n+\beta}.\]
\end{lemma}
\begin{proof}
We proceed by induction on $n$. For $n=1$, we may calculate:
\begin{align*}
t(P_{0,1,\beta},W)&= \frac{1}{|V(W)|^{2+\beta}}\sum_{v\in V(W)}d_v^{1+\beta}
\\ &= \frac{1}{|V(W)|^{1+\beta}}(\frac{1}{|V(W)|}\sum_{v\in V(W)}d_v^{1+\beta})
\\ &\ge \frac{1}{|V(W)|^{1+\beta}}(\frac{1}{|V(W)|}\sum_{v\in V(W)}d_v)^{1+\beta}
\\ &= (\frac{1}{|V(W)|^2}\sum_{v\in V(W)}d_v)^{1+\beta}
\\ &= t(P_1, W)^{1+\beta}
\end{align*}
where we have used the power mean inequality. 

For the inductive step, suppose $r\ge 2$, and that the statement holds for all $r'<r$. We take two cases. 

If $r$ is even, then due to Proposition \ref{hold} and Cauchy-Schwarz,
\begin{align*}
t(P_{0,r,\beta}, W)&\ge t(P_{\frac{\beta}{2}, r, \frac{\beta}{2}}, W)
\\ &\ge t(P_{0, \frac{r}{2}, \frac{\beta}{2}}, W)^2
\\ &\ge t(P_1, W)^{r+\beta}
\end{align*}
where we've used the inductive hypothesis in the last step. 

Similarly, if $r$ is odd,
\begin{align*}
t(P_{0,r,\beta}, W)&= t(P_{1, r-1, \beta}, W)
\\ &\ge t(P_{\frac{1+\beta}{2}, r-1, \frac{1+\beta}{2}}, W)
\\ &\ge t(P_{0, \frac{r-1}{2}, \frac{1+\beta}{2}}, W)^2
\\ &\ge t(P_1, W)^{r+\beta}.
\end{align*}

This completes the inductive step, so we are done.
\end{proof}

\begin{remark}
Taking $\beta=0$ in the previous lemma yields the Blakely-Roy inequality, $t(P_n,W)\ge t(P_1,W)^{n}$. 
\end{remark}

\begin{lemma}
Suppose $k$ is a positive integer, and $G_{k,\ell}$ is the graph on $2k+\ell+1$ vertices consisting of a cycle with length $2k+1$, with one of its vertices connected to an end of a path on $\ell$ edges. Then 
\[t(G_{k,\ell}, W)\ge t(C_{2k+1}, W)^{1+\frac{1}{2}\left\lceil\frac{\ell}{k}\right\rceil}.\]
\end{lemma}
\begin{proof}
Set $r=\left\lceil\frac{\ell}{2k}\right\rceil$, and $s=2k\lceil\frac{\ell}{2k}\rceil-\ell$. Define a graph $X$ as follows. If $0\le s\le k-1$, then $X$ consists of a $P_{r+1}$, with an additional length $2k$ path added between every pair of adjacent vertices in this base path. If $k\le s\le 2k-1$, then $X$ consists of a $P_r$, with an additional length $2k$ path between every pair of adjacent vertices in the base path, and with an additional length $k$ path attached at one end of the base path. Observe that in either case, $X$ contains $G_{k, \ell}$ as a subgraph. Let $W'$ denote the auxiliary weighted graph, defined by $W(a,b)$ equals the number of ordered cycles in $W$ containing $a,b$ as their first and second vertices, respectively.

We take two cases. First, suppose that $0\le s\le k-1$. Making use of the Lemma \ref{blakeroy}, we have
\begin{align*}t(G_{k, \ell}, W)&\ge t(X, W) 
\\&= t(P_{r+1}, W')
\\ &\ge t(P_1,W')^{r+1}
\\ &= t(C_{2k+1}, W)^{r+1}
\\ &=t(C_{2k+1},W)^{1+\lceil\frac{\ell}{2k}\rceil}
\\ &=t(C_{2k+1},W)^{1+\frac{1}{2}\lceil\frac{\ell}{2k}\rceil}.
\end{align*}
Here, we've used the fact that since the difference between $\ell$ and $\ell-2k\lceil\frac{\ell}{2k}\rceil$ is at most $k-1$, it holds that $\lceil\frac{\ell}{k}\rceil=2\lceil\frac{\ell}{2k}\rceil.$

If $r$ is odd, we take a similar approach. We first note that for any given vertex, the number of $P_k$s which start at that vertex is at least the square root of the number of $C_{2k+1}$s containing that vertex, since the square of the first quantity is the number of $P_{2k}$s which have the given vertex as a middle vertex. This means that in the odd case, $t(X, W)\ge t(P_{0,r,\frac{1}{2}}, W')$. Using this, we calculate as before.
\begin{align*}t(G_{k, \ell}, W)&\ge t(X, W) 
\\&= t(P_{0,r,\frac{1}{2}}, W')
\\ &\ge t(P_1,W')^{r+1}
\\ &= t(C_{2k+1}, W)^{r+1}
\\ &=t(C_{2k+1},W)^{1+\lceil\frac{\ell}{2k}\rceil}
\\ &=t(C_{2k+1},W)^{1+\frac{1}{2}\lceil\frac{\ell}{2k}\rceil}.
\end{align*}
\end{proof}

Now we proceed to the proof of the upper bound. Consider the following auxiliary graph $Y$: start with a cycle $v_1,v_2\ldots, v_{2k+2n}$ on $2k+2n$ vertices, and additionally connect $v_1$ to $v_{2k+1}$ and $v_{2n+1}$. Note that $Y$ contains a $C_{2n+1}$. Using the notation from the previous lemma, an application of \ref{hold} yields:
\begin{align*}t(C_{2n+1},W)&\ge t(Y,W)
\\ &\ge t(G_{k, \ell}, W)^2
\\ &\ge t(C_{2k+1}, W)^{2+\left\lceil\frac{n-k}{k}\right\rceil}
\\ &=t(C_{2k+1}, W)^{1+\left\lceil\frac{n}{k}\right\rceil}
\end{align*}
as desired.
\end{proof}

And we can ask a similar question here.
\begin{question}
Is it true that, for positive integers $n>k$, $\rho(C_{2k+1}, C_{2n+1})=\frac{n}{k}+o_{k,n}(1)$?
\end{question}

\section{Complete and Complete Multipartite Graphs}
We first examine the case where both $G=K_s, H=K_t$ are complete graphs. According to Proposition \ref{cond}, this quantity is finite exactly when $s\ge t$. In this case, $\rho(K_s, K_t)=\frac{t}{s}$ due to Kruskal-Katona \cite{kk}. 

We switch our attention now to complete multipartite graphs. We'll mostly use the convention that indices are always written in nonincreasing order. For a positive integer sequence $\textbf{a}=\{a_1,\ldots, a_n\}$, we write $K_{\bf{a}}$ as shorthand for $K_{a_1,\ldots, a_n}$. 

First, two quick lemmas:
\begin{lemma}\label{center}
Suppose $a,b>1$ are integers. Then, for any graphon $W$, 
\[t(K_{a+1,b-1}, W)t(K_{a-1, b+1}, W)\ge t(K_{a,b}, W)^2.\]
\end{lemma}
\begin{proof}
Let $K'_{a,b}$ denote the graph obtained by removing an edge from $K_{a,b}$. Then, as a consequence of Proposition \ref{hold}, we have
\[t(K_{a+1,b-1}, W)t(K_{a-1, b+1}, W)\ge t(K'_{a,b}, W)^2\ge t(K_{a,b},W)^2\]
as desired.
\end{proof}

For this next lemma, we first define $t(K_{n, x},W)$ where $n$ is a positive integer, $W$ is a graph, and $x$ is a nonnegative real. Namely,
\[t(K_{n,x},W):=\frac{1}{|V(W)|^{n+x}}\sum_{v_1,\ldots, v_n\in V(W)} d_{\{v_1,\ldots, v_n\}}^x.\]
This matches up with the usual definition when $x$ is a nonnegative integer. This allows us to define $\rho(G,H)$ for some of $G,H$ equal to such a $K_{n,x}$.

\begin{lemma}\label{entropy}
Suppose that $a,b, c$ be positive integers with $c>a$. Then
\[\rho(K_{c, \frac{bc}{a}}, K_{a,b})=\frac{c}{a}.\]
\end{lemma}
\begin{proof}
Taking $W$ to be a single looped vertex in a large, otherwise empty graph yields the required lower bound. 

For the upper bound, we make use of a discrete weighted version of Shearer's entropy lemma, found in \cite{friedgut2004hypergraphs}. Using this, and taking indices modulo $c$ we may calculate:
\begin{align*}
t(K_{c, \frac{bc}{a}}, W)&= \frac{1}{|V(W)|^{c+\frac{cb}{a}}}\sum_{v_1,\ldots, v_c\in V(W)} d_{\{v_1,\ldots, v_c\}}^{\frac{cb}{a}}.
\\ &\le \frac{1}{|V(W)|^{c+\frac{cb}{a}}}\sum_{v_1,\ldots, v_c\in V(W)} \prod_{i=1}^cd_{\{v_i,\ldots, v_{i+a-1}\}}^{\frac{b}{a}}.
\\ &\le (\frac{1}{|V(W)|^{a+b}}\sum_{v_1,\ldots, v_a\in V(W)} d_{\{v_1,\ldots, v_{a}\}}^b)^{\frac{c}{a}}
\\ &= t(K_{a,b}, W)^{\frac{c}{a}}
\end{align*}
as desired.
\end{proof}

Now we characterize $\rho(G,H)$ for most complete bipartite graphs $G,H$.

\begin{theorem}
Let $a_1\ge a_2, b_1\ge b_2$ be positive integers. Then
\[\rho(K_{a_1, a_2}, K_{b_1, b_2})=\begin{cases}
\frac{b_1b_2}{a_1a_2} & b_1\ge a_1, b_2\ge a_2
\\ \frac{b_2}{a_2} & b_1\le a_1, b_2\ge a_2
\\ \max\{\frac{b_2}{a_2}, \frac{b_1+b_2}{a_1+a_2}\} & b_1\le a_1, b_2\le a_2
\\ \frac{b_1+b_2}{a_1+a_2} & b_1\ge a_1, b_2\le a_2, b_1+b_2\le a_1+a_2
\\ \frac{b_1}{a_1+a_2-1} & b_1\ge a_1, b_2=1, b_1+b_2\ge a_1+a_2
\end{cases}.\]
\end{theorem}
\begin{proof}
As usual, the lower bounds are all cases of Proposition \ref{lowb}. In particular, the first bound follows from (a), the second through fourth follow from (b) and (d), and the fifth follows from (c).  

We turn our attention to the upper bounds. Suppose first that $b_1\ge a_1$ and $b_2\ge a_2$. Define $K_{i, 0}=K_{0,i}$ to be the empty graph on $i$ vertices, so that $t(K_{i,0}, W)=1$ for any $t, W$. Then, using Proposition \ref{hold} twice,
\begin{align*}
t(K_{b_1,b_2}, W)^{a_1a_2}&= (t(K_{b_1,b_2},W)^a_1t(K_{0,b_2},W)^{b_1-a_1})^{a_2}
\\ &\ge t(K_{a_1, b_2},W)^{b_1a_2}
\\ &= (t(K_{a_1,b_2},W)^{a_2}t(K_{a_1,0},W)^{b_2-a_2})^{b_1}
\\ &\ge t(K_{a_1,a_2}, W)^{b_1b_2}
\end{align*}
as desired.

The second upper bound (for the case $b_1\le a_1, b_2\ge a_2$)  follows from the first as follows.
\[t(K_{b_1,b_2},W)\ge t(K_{b_1, a_2},W)^{\frac{b_2}{a_2}}\ge t(K_{a_1, a_2},W)^{\frac{b_2}{a_2}},\]
where the last line holds because $K_{b_1, a_2}$ is a subgraph of $K_{a_1, a_2}$.

For the third upper bound, we take two cases. If $\frac{b_2}{a_2}\ge \frac{b_1}{a_1}$, then the maximum is equal to $\frac{b_2}{a_2}$. In this case, we find that
\[t(K_{b_1,b_2},W)\ge t(K_{\frac{b_1a_2}{b_2}, a_2}, W)^{\frac{b_2}{a_2}}\ge t(K_{a_1, a_2}, W)^{\frac{b_2}{a_2}}\]
due to Lemma \ref{entropy}. 

On the other hand, if $\frac{b_2}{a_2}\le \frac{b_1}{a_1}$, then the maximum is $\frac{b_1+b_2}{a_1+a_2}$. In this case, we can calculate:
\begin{align*}
t(K_{b_1,b_2}, W)^{(a_1+a_2)(\frac{a_1}{b_2}-\frac{a_1}{b_1})}&=t(K_{b_1,b_2},W)^{\frac{a_1^2}{b_2}-\frac{a_1a_2}{b_1}}t(K_{b_1,b_2},W)^{\frac{a_1a_2}{b_2}-\frac{a_1^2}{b_1}}
\\ &\ge t(K_{a_1, \frac{a_1b_2}{b_1}}, W)^{\frac{a_1b_1}{b_2}-a_2}t(K_{a_1, \frac{a_1b_1}{b_2}}, W)^{a_2-\frac{a_1b_2}{b_1}}
\\ &\ge t(K_{a_1, a_2}, W)^{(b_1+b_2)(\frac{a_1}{b_2}-\frac{a_1}{b_1})}
\end{align*}
as desired. 

Now we turn our focus to the fourth upper bound. The approach here is to reduce to the previous case via Lemma \ref{center}. Indeed,
\[t(K_{b_1,b_2},W)\ge t(K_{b_1, a_1+a_2-b_1}, W)^{\frac{b_1+b_2}{a_1+a_2}}\ge t(K_{a_1, a_2}, W)^{\frac{b_1+b_2}{a_1+a_2}}.\]

Finally, the fifth upper bound follows directly from Proposition \ref{starry}.
\end{proof}

There's one case remaining. We conjecture the right answer for this missing case.
\begin{conjecture}
Let $a_1\ge a_2, b_1\ge b_2$ be positive integers with $a_2\ge b_2>1, a_1\le b_1, b_1+b_2\ge a_1+a_2$. Then
\[\rho(K_{a_1,a_2},K_{b_1,b_2})=\max\{\frac{b_1b_2}{a_1a_2}, \frac{b_1+b_2-1}{a_1+a_2-1}\}.\]
\end{conjecture}

We can extend these results partially to the general multipartite case via the following result:
\begin{theorem}
Let $\bf{a}, \bf{b}$ be length $k$ sequences of positive integers such that $a\succ b$. Then
\[\rho(K_{\bf b}, K_{\bf a})=1.\]
Furthermore, if $||\bf a||_1=||\bf b||_1$, then this equality holds if and only if $a\succ b$.
\end{theorem}
\begin{proof}
For the lower bound, take construction (b) from Proposition \ref{lowb}.

For the upper bound, we prove a lemma:
\begin{lemma}\label{major}
Suppose $\bf a,\bf b$ are length $k$ positive integer vectors with $\bf a\succ\bf b$. Then there exists a sequence $\bf a=\bf a_0, \bf a_1, \ldots, \bf a_\ell=\bf b$ such that, for each $0\le i\le \ell-1$, the vectors $\bf a_i$ and $\bf a_{i+1}$ differ in at most two coordinates, and $\bf a_i\succ \bf a_{i+1}$.  
\end{lemma}
\begin{proof}

We induct on $\sum_{i=1}^k(\sum_{j=1}^i(a_j-b_j))$. Note that due to the majorization condition, this quantity is always nonnegative. For the base case, if the summation is zero, then $\bf a=\bf b$ and the result is clear.

Otherwise, $\bf a\neq \bf b$, which means that they differ in at least two indices. Let $r$ be the lowest index for which $b_r\neq a_r$, and let $s>r$ be the greatest index for which $b_s=b_{r+1}$. Let $\bf{b'}$ be the nonnegative integer vector obtained by adding one to the $r$th index of $\bf b$ and subtracting one from the $s$th index. Note that the indices of $\bf{b'}$ are still in nonincreasing order. We claim $\bf a\succ \bf{b'}\succ \bf{b}$. The latter majorization is clear. For the former, we just need to prove the inequality $\sum_{j=1}^i (a_j-b_j)\ge 0$ for each $1\le j\le k$. We take three cases.

If $i\le r-1$ or $i\ge s$, then $\sum_{j=1}^i b_j=\sum_{j=1}^i b_j'$, and so the inequality follows from the majorization $\bf a\succ\bf b$. 

If $i=r$, then $\bf b_i'\le \bf a_i$ and $\bf b'_j=\bf b_j=\bf a_j$ for $j<i$, so the result follows. 

For $r+1\le i\le s-1$, we argue via contradiction. Suppose that the inequality does not hold. Since the summation shifts by $1$ from $\bf b$ to $\bf b'$ in this regime and $\bf a\succ \bf b$, it must be the case that $\sum_{j=1}^ia_j=\sum_{j=1}^i b_j$. Since $a_r>b_r$ and $a_j=b_j$ for $j<r$, there exists some index $\ell\in [r+1, j]$ such that $a_\ell<b_\ell$. But then $a_i\le a_\ell<b_\ell=b_i$. Hence $\sum_{j=1}^{i-1}a_j<\sum_{j=1}^{i-1} b_j$, contradicting $\bf a\succ \bf b$. Hence, $\bf a\succ \bf{b'}$ as required. 

Now $\bf{b'}$ and $\bf{b}$ differ in at most two indices, and $\sum_{i=1}^k(\sum_{j=1}^i(a_j-b'_j))<\sum_{i=1}^k(\sum_{j=1}^i(a_j-b_j))$, so by the inductive hypothesis we can find a chain from $\bf a$ to $\bf {b'}$. This completes the inductive step, so we're done. 
\end{proof}

Due to the lemma, to prove the upper bound, it suffices to check that $t(K_{\bf a},W)\ge t(K_{\bf b}, W)$ when $\bf a\succ \bf b$ and $\bf a,\bf b$ differ in only two indices. This follows directly from an application of Lemma \ref{center}, so we are done. 

For the converse, suppose that the majorization fails at index $i$; that is, $\sum_{j=1}^ia_j<\sum_{j=1}^ib_j$. Then take $W$ to be a complete $k$-partite graph with $i$ parts having size $n$ for large $n$ and the remaining $k-i$ parts having size $1$. Then $t(\bf {a}, W)$ is $\Theta(n^{\sum_{j=1}^ia_j-||\bf a||_1})$, while $t(\bf {b}, W)$ is $\Theta(n^{\sum_{j=1}^ib_j-||\bf b||_1})$, so $\rho(K_{\bf b}, K_{\bf a})\ge \frac{\log t(\bf {a}, W)}{\log t(\bf {b}, W)}>1$ for large enough $n$, as desired.

\end{proof}

Denote by $K'_{a_1, \ldots, a_n}$ the $n$-partite graph $G$ defined as follows. $G$ contains a central clique $C$, with vertices $v_1, \ldots, v_n$, and for each $i$, $G$ contains $a_i$ vertices outside of $C$ which have neighborhood $C-\{v_i\}$. 

\begin{theorem}\label{hyper}
For any nonnegative integer vector $\bf{a}=\{a_1,\ldots, a_n\}$,
\[\rho(K_n, K'_{\bf{a}})=1+||\bf{a}||_1.\]
\end{theorem}

\begin{proof}
The lower bound construction can be found in \cite{gowers}. Indeed, for each $n$, this paper gives a family of graphs on $N$ vertices for $N$ sufficiently large which have $K_n$ density $N^{-1+o(1)}$ and $K'_{\bf{a}}$ density $N^{-1-||\bf{a}||_1+o(1)}$.

We turn to the upper bound; let $W$ be any graph. Let $\mathcal{R}$ denote the family of ordered vertex subsets of size $n$ made from the vertices in $W$ which form a $K_n$.  Then
\begin{align*}
t(K'_{\bf{a}}, W)&= \frac{1}{|V(W)|^{||\bf{a}||_1+n}}\sum_{S\in \mathcal{R}\atop S=\{v_1,\ldots v_n\}} \prod_{i=1}^nd_{S\sm v_i}^{a_i}
\\ &\ge \frac{1}{|V(W)|^{||\bf{a}||_1+n}}(\sum_{S\in \mathcal{R}\atop S=\{v_1,\ldots v_n\}} \prod_{i=1}^nd_{S\sm v_i}^{a_i})t(K_{n-1}, W)^{||\bf{a}||_1}
\\ &=\frac{1}{|V(W)|^{n(||\bf{a}||_1+1)}}(\sum_{S\in \mathcal{R}\atop S=\{v_1,\ldots v_n\}} \prod_{i=1}^nd_{S\sm v_i}^{a_i})\prod_{i=1}^n(\sum_{S\in \mathcal{R}\atop S=\{v_1,\ldots v_n\}}d_{S\sm v_i}^{-1})^{a_i}
\\ &\ge (\frac{1}{|V(W)|^n}\sum_{S\in \mathcal{R}} 1)^{1+||\bf{a}||_1}
\\ &= t(K_n, W)^{1+||\bf{a}||_1}
\end{align*}
as desired. 
\end{proof}

\begin{remark}
The above theorem is also true if the graphs are replaced with $n$-uniform hypergraphs where every copy of $K_n$ induces a hyperedge. The lower bound construction in this case is much easier; just take $W$ to be a uniformly pseudorandom hypergraph. 
\end{remark}

\section{Open Problems}
Here, we summarize some of the remaining open problems, including some minimal open values of $\rho(G,H)$.
\begin{question}What is the value of $\rho(C_3,C_4)$?
\end{question}
 We know $\rho(C_3, C_4)\in [\frac{3}{2}, \frac{8}{5}].$ The lower bound is given by Proposition \ref{lowb}, construction c.  The upper bound follows from the following inequality chain:
\begin{align*}t(C_4, W)^5&=t(\squ, W)^5
\\&\ge t(\edge, W)^4t(\squ, W)^4
\\&\ge t(\edge, W)^4t(\squz, W)^4
\\&\ge t(\tri, W)^8
\\ &= t(C_3, W)^8.
\end{align*}
The first inequality is due to $C_4$ being a Sidorenko graph, while the third is an application of Proposition \ref{hold}.  

\begin{question} What is the value of $\rho(P_5,P_3)$? \end{question}
 We know $\rho(P_5,P_3)\in [\frac{7}{10}, \frac{3}{4}].$ The lower bound is the construction from \cite{kopparty2011homomorphism}. The upper bound follows from the following inequality chain:
\[t(P_3, W)^4\ge t(P_{15}, W)\ge t(P_5, W)^3.\]
The first inequality is true since $P_{15}$ contains four vertex-disjoint copies of $P_3$, while the second is an instance of Theorem \ref{paths}.

The following question is a natural one, given the nature of the results found here.
\begin{question}
For graphs $G,H$ for which $\rho(G,H)$ exists and is finite, is it always rational? Algebraic?
\end{question}

Given the existence of Proposition 2.7 and other anomalies, it seems likely that a finite number of sequences cannot cover all cases. That is, the following seems likely to hold true.
\begin{question}
Is it true that, for every finite family of graphon sequences $\mathcal{F}$, there exists some graphs $G,H$ for which $\rho(G,H)$ is finite and $\frac{\log t(H,W_n)}{\log t(G,W_n)}$ does not converge to $\rho(G,H)$ as $n\to\infty$ for any $\{W_n\}\in \mathcal{F}$? 
\end{question}

Finally, Theorem \ref{hyper} motivates the following general question. 

\begin{question}
Let $H$ be any $n$-partite $n$-uniform hypergraph, and let $H^*$ denote the $2$-graph where a pair of vertices is connected by an edge in $H^*$ if and only if they are both in a hyperedge of $H$.  Is it always the case that $\rho(K_n, H)=\rho(K_n^*, H^*)$?
\end{question}
\nocite{*}

\end{document}